\documentclass[conference, 10pt]{IEEEtran}
\IEEEoverridecommandlockouts

\usepackage{amsmath,amssymb,amsthm,mathtools,cite,balance}

\def\B { \mathcal{B} }
\def\L { \mathcal{L} }
\def\V { \mathcal{V} }
\def\R { \mathbb{R}  }
\def\N { \mathbb{N}  }
\def\C { \mathbb{C}  }
\def\Q { \mathbb{Q}  }
\def\comma{\text{, }}

\newtheorem{theorem}{Theorem}

\theoremstyle{definition}

\theoremstyle{remark}

\begin{document}  \onecolumn

\title{Transient Signal Spaces and Decompositions}
\author{\IEEEauthorblockN{Tarek A. Lahlou\IEEEauthorrefmark{1} and Anuran Makur\IEEEauthorrefmark{2}}
\IEEEauthorblockA{\IEEEauthorrefmark{1}Digital Signal Processing Group, \IEEEauthorrefmark{2}Claude E. Shannon Communication and Network Group}
\IEEEauthorblockA{Massachusetts Institute of Technology}}

\bstctlcite{IEEEexample:BSTcontrol}

\maketitle

\begin{abstract}
	In this paper, we study the problem of transient signal analysis. A signal-dependent algorithm is proposed which sequentially identifies the countable sets of decay rates and expansion coefficients present in a given signal. We qualitatively compare our method to existing techniques such as orthogonal exponential transforms generated from orthogonal polynomial classes. The presented algorithm has immediate utility to signal processing applications wherein the decay rates and expansion coefficients associated with a transient signal convey information. We also provide a functional interpretation of our parameter extraction method via signal approximation using monomials over the unit interval from the perspective of biorthogonal constraint satisfaction.
\end{abstract}

\tableofcontents

\section{Introduction} \label{sec:Introduction}
	
	Transient signal decomposition is a widely studied classical problem in signal processing. It concerns decomposing transient signals in terms of a countable set of decaying exponential signals. Analytically, this problem is trivially solved by the inverse Laplace transform. To observe this, consider a transient signal, $x:[0\comma\infty)  \rightarrow  \R$, defined by:
	\begin{equation} \label{Eq:Transient Signal}
		\forall t \geq 0\comma \enspace x(t) = \sum_{n \in \N}{\alpha_n e^{-\lambda_n t}}
	\end{equation}
	where $\left\{\lambda_n \in \left(0\comma\infty\right) : n \in \N\right\}$ is the set of decay rates of $x$, $\left\{\alpha_n \in \R : n \in \N\right\}$ is the set of corresponding expansion coefficients of $x$, and $\N \triangleq \left\{1\comma2\comma3\comma\dots\right\}$ denotes the set of natural numbers. The equality in \eqref{Eq:Transient Signal} holds in the sense of pointwise convergence of functions. Hence, we may recast it more formally as: 
	\begin{equation} \label{Eq:Pointwise Convergence}
		 \forall t \geq 0\comma \enspace x(t) = \lim_{k \rightarrow \infty}{\sum_{n = 1}^{k}{\alpha_n e^{-\lambda_n t}}} .  
	\end{equation}
	In the ensuing discussion, all equalities involving infinite summations of functions will refer to pointwise convergence as illustrated in \eqref{Eq:Pointwise Convergence}. 

	The (unilateral or one-sided) Laplace transform is a ubiquitous tool in applied mathematics and engineering. In systems theory, it is particularly useful in solving various differential equations with known initial conditions. We briefly introduce the Laplace transform, and refer readers to \cite{LaplaceTransform} for a rigorous and complete development of the subject. In addition, the pertinent measure and integration theory used in our discourse can be found in \cite{SteinRealAnalysis}. For a Borel measurable function $f:\left([0\comma\infty)\comma\B\left([0\comma\infty)\right)\right) \rightarrow \left(\R\comma\B\left(\R\right)\right)$, where $\B\left(X\right)$ denotes the Borel $\sigma$-algebra on a topological space $X$, the Laplace transform of $f$, denoted $F:\text{dom}\left(F\right) \rightarrow \C$, is defined as:
	\begin{equation} \label{Eq:Laplace Transform}
		\forall s \in \text{dom}\left(F\right)\comma \enspace F\left(s\right) \triangleq \int_{\left[0\comma\infty\right)}{f(t)e^{-st} \, d\mu(t)}
	\end{equation} 
	where $\mu$ denotes the Lebesgue measure and the integral is the Lebesgue integral. The domain of $F$, which is also known as the region of convergence (ROC) of the Laplace transform, is given by:
	\begin{equation*}
	 	\text{dom}\left(F\right) \triangleq \left\{s \in \C: \int_{\left[0\comma\infty\right)}{\left|f(t)e^{-st}\right| \, d\mu(t)} < \infty \right\}
	\end{equation*}
	which is the set of all $s \in \C$ such that $f(t)e^{-st}$ is Lebesgue integrable and the Laplace integral in \eqref{Eq:Laplace Transform} is finite. Observe that \eqref{Eq:Laplace Transform} can be construed as a decomposition of $F$ in terms of a ``basis'' of decaying exponential functions of $s$, $\left\{e^{-st}:t \in (0\comma\infty) \right\}$ (where we do not claim that $\left\{e^{-st}:t \in (0\comma\infty) \right\}$ is a Hamel basis in the linear algebraic sense or a Schauder basis in the Banach space sense \cite{BanachSpaceTheory}). With this interpretation, the action of the corresponding ``biorthogonal basis'' is given by the inverse Laplace transform. In other words, the inverse Laplace transform produces the expansion coefficients of $F$ corresponding to different decaying exponential ``basis'' functions. Guided by this intuition, we can construct:
	\begin{equation}
		\forall s \in \C\comma \enspace F_x\left(s\right) = \sum_{n \in \N}{\alpha_n e^{-\lambda_n s}}
	\end{equation}
	corresponding to the transient signal $x$ in \eqref{Eq:Transient Signal}. The inverse Laplace transform of $F_x:\C \rightarrow \C$ can be verified to be the generalized function (distribution or signed measure):
	\begin{equation} \label{Eq:Inverse Laplace Transform}
		f_x\left(t\right) = \sum_{n \in \N}{\alpha_n \delta\left(t - \lambda_n\right)}
	\end{equation}
	where $\delta\left(\cdot\right)$ denotes the Dirac delta function. For example, in the case where $\forall n \in \N\comma \enspace \alpha_n \geq 0$, we can prove \eqref{Eq:Inverse Laplace Transform} by first letting $s = i\omega$, $i \triangleq \sqrt{-1}$, $\omega \in \R$, to get a Fourier transform $F_x\left(i\omega\right)$, and then exploiting convergence results of inverse Fourier transforms (characteristic functions), such as L\'{e}vy's continuity theorem \cite{CinlarProbabilityTheory}, to deduce the weak convergence of the finite measures corresponding to the partial summations of \eqref{Eq:Inverse Laplace Transform}. By inspecting \eqref{Eq:Inverse Laplace Transform}, we can read off the decay rates and expansion coefficients of the transient signal $x$. Moreover, under appropriate regularity conditions, the uniqueness of the inverse Laplace transform ensures that transient signals have unique sets of decay rates and expansion coefficients. 

	Unfortunately, the inverse Laplace transform is numerically infeasible to compute, although it elegantly identifies the decay rates and expansion coefficients of transient signals. As a result, given a transient signal $x$ in \eqref{Eq:Transient Signal}, we are interested in identifying its decay rates $\left\{\lambda_n \in \left(0\comma\infty\right) : n \in \N\right\}$ and then extracting the corresponding expansion coefficients $\left\{\alpha_n \in \R : n \in \N\right\}$ in a computationally efficient manner. To this end, we first survey some well-known algorithms to solve this problem, and then present an alternative algorithm which overcomes some of the shortcomings faced by current methods.

\section{Previous Approaches}
	We now briefly review some well-established alternative methods from the literature dealing with both discrete and continuous exponential signal models, respectively.
	
	\subsection{Prony's Method and Descendants}
		Fitting a finite time series to a parametric model consisting of the sum of finitely many exponential terms underlies Prony's original method and forms the basis of its descendants \cite{Prony1, Prony2, Prony3, Prony4, Prony5, Prony6, Prony7, Prony8, Prony9}. Indeed, such methods typically rely in some key way on the relationship between the chosen signal model and the homogeneous solution to a difference equation in order to identify the pertinent decay rates via the intermediary factoring of an associated characteristic polynomial. Prior knowledge that the roots of this polynomial are purely real is difficult to leverage for nonharmonic signal models. Moreover, even in a finite-dimensional discrete-time context, extracting the expansion coefficients for known decay rates requires solving an ill-conditioned Vandermonde system \cite{VanderEst} for which specialized algorithms only alleviate some numerical instabilities \cite{LahlouOppenheim}. Methods of this class have been shown to be misaligned with maximum-likelihood estimates of exponential signal parameters in the presence of noise \cite{Bresler}. Pencil methods have reported improved robustness to noise due in part to a hybrid singular value decomposition and Prony-like approach \cite{MatrixPencil}.

	\subsection{Orthogonal Exponential Transforms} \label{Orthogonal Exponential Transforms}
		The approximation of arbitrary smooth functions over a finite interval using a real exponential power series is dissatisfying for primarily two reasons: the optimal approximation coefficients typically change with increasing model order, and the associated computations may require solving transcendental equations \cite{Gordon}. Alternatively, an orthogonal exponential transform constructed by translating an orthogonal polynomial class to the unit interval and performing a suitable change of variables is limited in that the set of possible decay rates $\left\{\lambda_n \in \left(0\comma\infty\right) : n \in \N\right\}$ is restricted to $\N$. 

		We now illustrate using Jacobi polynomials the construction of such an orthogonal exponential transform. The degree $n \geq 0$ Jacobi polynomial, $J_{n}^{(a\comma b)}  : [0\comma1]  \rightarrow  \R$, with degrees of freedom $a\comma b \in \R$, is given by the Rodrigues formula \cite{OrthogonalPolynomialsAskey, OrthogonalPolynomialsChihara, Morse53, MoreJacobi}: 
		\begin{equation} \label{Eq:Jacobi Polynomial}
	 		J_{n}^{(a\comma b)}(z) \triangleq \frac{\Gamma(b)z^{1-b}\left(1-z\right)^{b-a}}{\Gamma(b+n)}\frac{d^{n}}{dz^{n}}\left( z^{b+n-1}\left(1-z\right)^{a+n-b}\right)
	 	\end{equation}
	 	and further satisfies the recurrence relations \cite{OrthogonalPolynomialsAskey, OrthogonalPolynomialsChihara, Morse53, MoreJacobi}:
		\begin{align*}
	 		\frac{d}{dz}J_{n}^{(a\comma b)}(z) 	&= 	-\frac{n\left(n+a\right)}{b}J_{n-1}^{(a+1\comma b+1)}(z) \\
			zJ_{n}^{(a\comma b)}(z)				&= 	 \frac{b-1}{2n+a}\left(J_{n}^{(a-1\comma b-1)}(z)-J_{n+1}^{(a-1\comma b-1)}(z)\right)
	 	\end{align*}
	 	where $\Gamma(\cdot)$ denotes the gamma function. Degenerations of Jacobi polynomials commonly arising in signal processing contexts include Chebyshev $\left(a=b=-\frac{1}{2}\right)$ and Legendre ($a=b=0$) polynomials. The orthogonality of Jacobi polynomials on the closed unit interval, for $\{(a\comma b)\in\R^2 \colon a>0\comma  a+1>b\}$, is readily verified by checking that:
	 	\begin{equation*}
	 		\int_{[0\comma 1]} J_{m}^{(a\comma b)}(z)J_{n}^{(a\comma b)}(z)w(z)dz 
	 		 =  \frac{\Gamma(n)\Gamma^2\left(b\right)\Gamma\left(n+a-b+1\right)}{\left(a+2n\right)\Gamma\left(a+n\right)\Gamma\left(b+n\right)}\Delta[{m-n}]
	 	\end{equation*}
	 	where $\Delta[\cdot]$ is the Kronecker delta function, and $w \colon [0\comma 1] \rightarrow \R$ is the scaled beta density weighting function $w(z) = z^{b-1}\left(1-z\right)^{a-b}$.
	 	
	 	In translating Jacobi polynomials into a transient setting, we proceed using $a=b=2$ for simplicity (consistent with the presentations in \cite{OETFirst} and \cite{OET}) and set $z=e^{-t}$. Factoring and distributing the transformed weighting function into definition \eqref{Eq:Jacobi Polynomial}, the $n$th orthogonal exponential basis element, $\widetilde{J}_n:[0\comma \infty) \rightarrow \R\comma  n \in \N$, is defined by:
		\begin{equation}
	  		\forall t \geq 0\comma  \enspace \widetilde{J}_{n}(t) \triangleq \left(-1\right)^{n-1} \sqrt{2n^{3}}e^{-t}J_{n-1}^{(2\comma 2)}\left(e^{-t}\right).
	  \end{equation} 
	  	These polynomials in $e^{-t}$  form an orthogonal basis under the inner product later defined in \eqref{Standard Inner Product}. With an orthogonal basis in place, obtaining the expansion coefficients $\left\{\alpha_n \in \R : n \in \N\right\}$ for a given transient signal $x$ is a standard exercise. We simply decompose $x$ as:
		\begin{equation}
			 x(t) = \sum_{n \in \N}\left(\int_{[0\comma \infty)}{x\left(\tau\right)\widetilde{J}_{n}(\tau) \, d\tau} \right)\widetilde{J}_{n}(t)
		\end{equation}
		and after some straightforward manipulations, we obtain:
		\begin{equation}
			x(t) = \sum_{n \in \N}\sum_{k=1}^{n}\left( \int_{[0\comma \infty)}x\left(\tau\right)\widetilde{J}_{n}(\tau) \, d\tau \right)c_{n\comma k}e^{-kt} = \sum_{n \in \N} \alpha_{n}e^{-n t}
		\end{equation}
		where $c_{n\comma k}$ is the coefficient of the term $e^{-kt}$ in $\widetilde{J}_{n}(t)$. We conclude by remarking that selecting $a=b=2$ was not a limitation; in \cite{Armstrong}, an orthogonal exponential transform is designed with $a = 4$ and $b = 3$.

\section{Transient Signal Space and Its Properties}
	We now formally set up the space of transient signals. We assume that there is a fixed but unknown set of decay rates, $\Lambda \subset \left(0\comma \infty\right)$, that is countable and well-ordered. The well-ordering property of $\Lambda$ ensures that $\Lambda$ is a totally ordered set in which every non-empty subset contains a well-defined minimal element. This precludes sets akin to $\Q$ (the rational numbers). Moreover, the well-ordering endows the countable set $\Lambda = \left\{\lambda_n \in \left(0\comma \infty\right) : n \in \N\right\}$ with a natural enumeration: $0 < \lambda_1 < \lambda_2 < \cdots$, that we fix from hereon. This setup is more general than previous approaches. Indeed, the countable and well-ordering assumptions on $\Lambda$ subsume the case of $\Lambda \subseteq \N$, which is a necessary restriction made in the orthogonal polynomial approach in Subsection \ref{Orthogonal Exponential Transforms}.

	The space of transient signals, $\V(\Lambda)$, is parametrized by $\Lambda$, and is defined as the set of all functions $x:[0\comma \infty) \rightarrow \R$ satisfying \eqref{Eq:Transient Signal}, in a pointwise convergence sense, for some countable set of expansion coefficients $\left\{\alpha_n \in \R : n \in \N\right\}$ that is absolutely summable: 
	\begin{equation} \label{Eq:Absolute Summability Condition}
		\sum_{n \in \N}{\left|\alpha_n\right|} < \infty .
	\end{equation}
	The absolute summability of the expansion coefficients ensures that the limit in \eqref{Eq:Pointwise Convergence} converges to a finite real number for every $t \geq 0$, and permits us to use any arbitrary order of summation over $n \in \N$ in \eqref{Eq:Transient Signal}. We define $\L^2\left([0\comma \infty)\right)$ as the separable Hilbert space of Borel measurable real functions on $[0\comma \infty)$ equipped with the standard $\L^2\left([0\comma \infty)\right)$ inner product:
	\begin{equation} \label{Standard Inner Product}
		\forall f\comma g \in \L^2\left([0\comma \infty)\right)\comma  \enspace \left<f\comma g\right> \triangleq \int_{[0\comma \infty)}{f(t)g(t) \, d\mu(t)} .
	\end{equation}
	The next theorem presents some properties of the transient signal space $\V(\Lambda)$, and relates $\V(\Lambda)$ to $\L^2\left([0\comma \infty)\right)$.

	\begin{theorem}[Properties of Transient Signal Space] \label{Thm:Properties of Transient Signal Space}
		The transient signal space, $\V(\Lambda)$, has the following properties:
		\begin{enumerate} 
			\item Every signal $x \in \V(\Lambda)$ is Borel measurable and continuous.
			
			\item If each $x \in \V(\Lambda)$ represents the associated equivalence class of functions in $\L^2\left([0\comma \infty)\right)$ that are equal to $x$ $\mu$-almost everywhere, then $\V(\Lambda) \subsetneq \L^2\left([0\comma \infty)\right)$.
			
			\item $\V(\Lambda)$ is an inner product space equipped with the standard $\L^2\left([0\comma \infty)\right)$ inner product.
		\end{enumerate}
	\end{theorem}

	\begin{proof}
		Fix any transient signal $x \in \V(\Lambda)$ and let \eqref{Eq:Transient Signal} be its expansion. Then, $x$ is Borel measurable because it is the pointwise limit of continuous (and hence, Borel measurable) functions \cite{SteinRealAnalysis}. Moreover, observe that for any $k \in \N$:
		\begin{align*}
			\sup_{t \geq 0}{\left|x(t) - \sum_{n = 1}^{k}{\alpha_n e^{-\lambda_n t}}\right|} & = \sup_{t \geq 0}{\left|\lim_{m \rightarrow \infty}{\sum_{n = k+1}^{m}{\alpha_n e^{-\lambda_n t}}}\right|} \\
			& \leq \sup_{t \geq 0}{\sum_{n = k+1}^{\infty}{\left|\alpha_n\right| e^{-\lambda_n t}}} \\
			& = \sum_{n = k+1}^{\infty}{\left|\alpha_n\right|}
		\end{align*}  
		where the first equality follows \eqref{Eq:Transient Signal}, the second inequality follows from the continuity of $t \mapsto |t|$ and the triangle inequality, and the final equality holds because $\forall n \in \N\comma  \enspace e^{-\lambda_n t} \leq 1$ with equality if and only if $t = 0$. Hence, by letting $k \rightarrow \infty$, we see that \eqref{Eq:Transient Signal} holds in a uniform convergence sense due to \eqref{Eq:Absolute Summability Condition} \cite{Rudin}. Since the partial summations in \eqref{Eq:Pointwise Convergence} are continuous and converge uniformly to $x$, $x$ must be continuous by the uniform convergence theorem \cite{Rudin}. The continuity of $x$ also implies it is Borel measurable. This proves the first property. Now consider:
		\begin{align*}
			\int_{[0\comma \infty)}{x^2(t) \, d\mu(t)} & = \int_{[0\comma \infty)}{\sum_{n\comma m \in \N}{\alpha_n \alpha_m e^{-\left(\lambda_n + \lambda_m\right)t}} \, d\mu(t)} \\
			& \leq \int_{[0\comma \infty)}{\sum_{n\comma m \in \N}{\left|\alpha_n\right| \left|\alpha_m\right| e^{-\left(\lambda_n + \lambda_m\right)t}} \, d\mu(t)} \\
			& \leq \int_{[0\comma \infty)}{e^{-2\lambda_1 t} \sum_{n\comma m \in \N}{\left|\alpha_n\right| \left|\alpha_m\right|} \, d\mu(t)} < \infty
		\end{align*}
		where the first equality follows from  the Cauchy product corresponding to $x^2(t)$ after swapping the order of summations and changing dummy variables; these operations hold because the series in \eqref{Eq:Transient Signal} absolutely converges for every $t \geq 0$ due to \eqref{Eq:Absolute Summability Condition}, which implies that the Cauchy product absolutely converges for every $t \geq 0$ by Mertens' theorem and Exercise 13 in Chapter 3 of \cite{Rudin}. The second inequality follows from the triangle inequality, the continuity of $t \mapsto |t|$, and the monotonicity of the Lebesgue integral, and the third inequality follows from the well-ordered enumeration of $\Lambda$. The third expression is finite because $\lambda_1 > 0$ and the expansion coefficients satisfy \eqref{Eq:Absolute Summability Condition}. Hence, $x \in \V(\Lambda) \Rightarrow x \in \L^2\left([0\comma \infty)\right)$, which means that $\V(\Lambda) \subseteq \L^2\left([0\comma \infty)\right)$, where each continuous signal in $\V(\Lambda)$ (using the first property) uniquely represents the corresponding equivalence class of functions in $\L^2\left([0\comma \infty)\right)$ that are equal $\mu$-almost everywhere. Note that such equivalence classes are the vectors of $\L^2\left([0\comma \infty)\right)$ \cite{SteinRealAnalysis}, and any such equivalence class contains at most one continuous function. By a uniqueness of inverse Laplace transform argument as delineated in Section~\ref{sec:Introduction}, we can show that $e^{-\lambda t} \in \L^2\left([0\comma \infty)\right) \backslash \V(\Lambda)$ if $\lambda \notin \Lambda$ and $\lambda > 0$. This implies that $\V(\Lambda) \subsetneq \L^2\left([0\comma \infty)\right)$ (proper subset), which proves the second property.

		We now prove the third property. Recall that $\mathcal{C}\left([0\comma \infty)\right)$, the set of real continuous functions on $[0\comma \infty)$, is a vector space over $\R$ with equality defined pointwise, and $\V(\Lambda) \subsetneq \mathcal{C}\left([0\comma \infty)\right)$ by the first property. To verify that $\V(\Lambda)$ is a linear subspace of $\mathcal{C}\left([0\comma \infty)\right)$, we must check that $\V(\Lambda)$ contains the additive identity function, and is closed under addition of functions and multiplication by real scalars \cite{LinearAlgebraAxler}. The zero function (additive identity) is clearly in $\V(\Lambda)$. If $x \in \V(\Lambda)$ with expansion coefficients $\left\{\alpha_n \in \R : n \in \N\right\}$ and $y \in \V(\Lambda)$ with expansion coefficients $\left\{\beta_n \in \R : n \in \N\right\}$, then for every $a\comma b \in \R$, $\left\{a\alpha_n + b\beta_n \in \R : n \in \N\right\}$ is a valid set of expansion coefficients because it is absolutely summable by the triangle inequality. It is easily verified using \eqref{Eq:Transient Signal} that $ax + by$ has expansion coefficients $\left\{a\alpha_n + b\beta_n \in \R : n \in \N\right\}$, and hence, $ax+ by \in \V\left(\Lambda\right)$ and $\V\left(\Lambda\right)$ is a vector space over $\R$. By the second property, $\V\left(\Lambda\right)$ is a linear subspace of $\L^2\left([0\comma \infty)\right)$, and inherits the standard $\L^2\left([0\comma \infty)\right)$ inner product with the nuance that $\forall f \in \V\left(\Lambda\right)$, $\left<f\comma f\right> = 0$ implies that $f$ is the everywhere zero function. This subtlety arises because continuous functions in $\V\left(\Lambda\right)$ represent the corresponding equivalence classes of functions in $\L^2\left([0\comma \infty)\right)$. This proves the third property.
	\end{proof}

	In Theorem \ref{Thm:Properties of Transient Signal Space}, we are not concerned with whether $\V\left(\Lambda\right)$ is a closed linear subspace of $\L^2\left([0\comma \infty)\right)$ with respect to the standard $\L^2\left([0\comma \infty)\right)$-norm (which would imply that $\V\left(\Lambda\right)$ is complete, and therefore, a sub-Hilbert space of $\L^2\left([0\comma \infty)\right)$). This is because equality in $\V\left(\Lambda\right)$ is defined pointwise, whereas equality in $\V\left(\Lambda\right)$ when embedded in $\L^2\left([0\comma \infty)\right)$ is defined $\mu$-almost everywhere. The proof of correctness of our transient signal decomposition algorithm will rely critically on the pointwise convergence assumption. For this reason, we also do not pursue developing $\V\left(\Lambda\right)$ as the closed linear span of a Schauder basis of decaying exponential signals associated with $\Lambda$ \cite{BanachSpaceTheory}. We next illustrate some properties of transient signals that will naturally engender our algorithm as a corollary.

	\begin{theorem}[Properties of Transient Signals] \label{Thm:Properties of Transient Signals}
		Every transient signal $x \in \V\left(\Lambda\right)$ with expansion coefficients $\left\{\alpha_n \in \R : n \in \N\right\}$ satisfies:
		\begin{enumerate}
			\item Vanishing Property: $\displaystyle{\lim_{t \rightarrow \infty}{x(t)} = 0}$
			
			\item Coefficient Isolation Property: $\displaystyle{\lim_{t \rightarrow \infty}{e^{\lambda_1 t}x(t)} = \alpha_1}$
			
			\item Laplace Principle: $\displaystyle{\lim_{t \rightarrow \infty}{-\frac{1}{t}\log\left(\left|x(t)\right|\right)} = \lambda_1}$
		\end{enumerate}
		where $\log : (0\comma \infty) \rightarrow \R$ denotes the natural logarithm, and we assume without loss of generality that $\alpha_1 \neq 0$ for the third property.
	\end{theorem}

	\begin{proof}
		Observe from \eqref{Eq:Transient Signal} that for every $t \geq 0$:
		$$ \left|x(t)\right| = \left|\sum_{n \in \N}{\alpha_n e^{-\lambda_n t}}\right| \leq \sum_{n \in \N}{\left|\alpha_n\right| e^{-\lambda_n t}} \leq e^{-\lambda_1 t} \sum_{n \in \N}{\left|\alpha_n\right|} $$
		where the first inequality follows from the triangle inequality and the continuity of $t \mapsto |t|$, and the second inequality holds because $\Lambda$ is well-ordered. Since the summation in the rightmost expression is finite due to (\ref{Eq:Absolute Summability Condition}), letting $t \rightarrow \infty$ produces the first property. Next, using the same sequence of steps with the aforementioned justifications, we again observe from \eqref{Eq:Transient Signal} that for every $t \geq 0$:
		\begin{align*} 
			\left|e^{\lambda_1 t} x(t) - \alpha_1\right| = \left|\sum_{n = 2}^{\infty}{\alpha_n e^{-\left(\lambda_n - \lambda_1\right) t}}\right| & \leq \sum_{n = 2}^{\infty}{\left|\alpha_n\right| e^{-\left(\lambda_n - \lambda_1\right) t}} \\
			& \leq e^{-\left(\lambda_2 - \lambda_1\right) t} \sum_{n = 2}^{\infty}{\left|\alpha_n\right|}
		\end{align*}
		from which we deduce the second property by letting $t \rightarrow \infty$.

		The third property is a variant of the Laplace principle from large deviations theory (Lemma 1.2.15 in \cite{LDP}), which is useful in deriving upper bounds of large deviation principles. Informally, it states that the rate of decay of a sum of decaying exponential functions is dominated by the slowest rate. We now prove the third property. Once again, we observe from \eqref{Eq:Transient Signal} that for every $t \geq 0$:
		$$ \frac{1}{t}\log\left(\left|x(t)\right|\right) = -\lambda_1 + \frac{1}{t}\log\left(\left|\alpha_1 + \sum_{n = 2}^{\infty}{\alpha_n e^{-\left(\lambda_n - \lambda_1\right) t}}\right|\right) .  $$ 
		Furthermore, we have:
		$$ \lim_{t \rightarrow \infty}{\log\left(\left|\alpha_1 + \sum_{n = 2}^{\infty}{\alpha_n e^{-\left(\lambda_n - \lambda_1\right) t}}\right|\right)} = \log\left(\left|\alpha_1\right|\right) $$
		using the continuity of $t \mapsto \log\left(|t|\right)$, and the first property applied to the summation (which is a signal in $\V\left(\left\{\lambda_n - \lambda_1 : n \in \N \backslash \{1\}\right\}\right)$, where $\left\{\lambda_n - \lambda_1 : n \in \N \backslash \{1\}\right\}$ forms a valid set of decay rates as $\Lambda$ is well-ordered). Since $\log\left(\left|\alpha_1\right|\right)$ is finite by the assumption that $\alpha_1 \neq 0$, letting $t \rightarrow \infty$ proves the third property.
	\end{proof}

	In the next section, we use Theorem \ref{Thm:Properties of Transient Signals} to derive our algorithm for performing transient signal decomposition.

\section{Computing Decay Rates and Expansion Coefficients}
	We now propose an iterative procedure through which we can recover the decay rates and expansion coefficients of any transient signal $x \in \V\left(\Lambda\right)$. \\
	 \\ 
	\textbf{Transient Signal Decomposition Algorithm:} \\
	Let $x_k:[0\comma \infty) \rightarrow \R$ be the processed signal after the $(k-1)$th iteration of the algorithm, $k \in \N$, so that $x_1 = x$. At the $k$th iteration, perform the following steps: 
	\begin{enumerate}
		\item Compute $\displaystyle{\gamma_k = \lim_{t \rightarrow \infty}{-\frac{1}{t}\log\left(\left|x_k(t)\right|\right)}}$.
	
		\item Compute $\displaystyle{\beta_k = \lim_{t \rightarrow \infty}{e^{\gamma_k t}x_k(t)}}$.
	
		\item Set $x_{k+1}(t) = x_k(t) - \beta_k e^{-\gamma_k t}$ for every $t \geq 0$.
		
	\end{enumerate}
	Repeat forever.  \\
	\\
	We note that termination at step $K  \in  \N$ identifies the $K-1$ most dominant rates in \eqref{Eq:Transient Signal}; if $x_K = 0$, then $x$ contained at most $K-1$ non-zero terms. The next theorem asserts that $\Lambda = \left\{\gamma_k : k \in \N\right\}$ and $\left\{\beta_k : k \in \N\right\}$ is the set of expansion coefficients of $x$.

	\begin{theorem}[Correctness of Algorithm] 
		For any transient signal $x \in \V\left(\Lambda\right)$ with expansion coefficients $\left\{\alpha_k \in \R : k \in \N\right\}$ such that $\forall k \in \N$, $\alpha_k \neq 0$ without loss of generality, the transient signal decomposition algorithm produces the correct decay rates and expansion coefficients: $\forall k \in \N$, $\gamma_k = \lambda_k$ and $\beta_k = \alpha_k$.
	\end{theorem}

	\begin{proof} 
		The correctness of the transient signal decomposition algorithm follows inductively from the coefficient isolation property and Laplace principle in Theorem \ref{Thm:Properties of Transient Signals}. In particular, they justify the computations performed on $x_k \in \V\left(\Lambda \backslash \left\{\lambda_1\comma \dots\comma \lambda_{k-1}\right\}\right)$, $k \in \N$, at the $k$th iteration of the algorithm, where we interpret $\left\{\lambda_1\comma \dots\comma \lambda_{0}\right\}$ as the empty set. 
	\end{proof}

	We remark that an analogous algorithm can be designed for transient discrete-time signals and the rigorous development for the discrete case parallels the continuous case.

	\subsection{A Functional Interpretation of Expansion Coefficient Extraction} \label{Functional Interpretation}
		In the sequel, we focus on a particular method of expansion coefficient extraction for a given transient signal $x$ of the form \eqref{Eq:Transient Signal}, where $\Lambda$ is assumed known from the outset. A rigorous treatment of this method is outside the scope of this paper; our intent here is to provide insight into decomposing $x$ via the sequential application of appropriately defined linear functionals belonging to the dual space of $\V(\Lambda)$. To this end, we now state the integral operator property corresponding to mapping the so-called sifting property on the closed unit interval, $[0\comma 1]$, to the non-negative extended reals, $[0\comma \infty]$:
		\begin{equation}
			\int_{[0\comma 1]} \delta(z)x(z)dz = x(0) \xrightarrow{z = e^{-t}}	\int_{[0\comma \infty]} e^{-t} \delta \left( e^{-t} \right) x(t) dt \triangleq x(\infty)  \label{Sampling at Infinity}
		\end{equation}
		where the right hand side of this expression additionally serves as our definition for continuity at infinity since $[0\comma 1]$ is homeomorphic to $[0\comma \infty]$ with the translated topology. 

		We next assert that the transient signal $x$ is decomposable as:
		\begin{equation}
			x(t) = \sum_{n \in \N} R_{n}(x)e^{-\lambda_{n} t}
		\end{equation}
		where the linear functionals, $R_{n} \colon \V(\Lambda) \rightarrow \R$, $n \in \N$, are:
		\begin{equation} \label{Eq:R Definition}
			R_{n}(x) \triangleq \int_{[0\comma \infty]}  e^{-t}\delta\left(e^{-t}\right)e^{\lambda_{n} t}\left(x(t) -   \sum_{k = 1}^{n-1}  R_{k}(x)e^{-\lambda_{k}t}\right) dt.
		\end{equation}
		We sketch a justification of this claim (and hence the claim that $\forall n \in \N\comma  \enspace \alpha_{n} = R_{n}(x)$), by arguing that the linear functionals satisfy a set of conditions resembling biorthogonality: 
		\begin{equation} \label{Eq:Transient Biorthogonal}
			\forall n\comma k \in \N\comma  \enspace R_{n}\left(e^{-\lambda_kt}\right) = \Delta[n-k] = \left\{ \begin{array}{cc}1\comma  & n=k\\0\comma  & n\neq k\end{array}\right. .
		\end{equation}
		The possible scenarios of \eqref{Eq:Transient Biorthogonal} are summarized next using:
		\begin{equation*}
			R_{n}\left(e^{-\lambda_k t}\right) = \int_{[0\comma \infty]}e^{-\tau}\delta\left(e^{-\tau}\right) e^{\lambda_n \tau}  \left( e^{-\lambda_k \tau} - \sum_{j=1}^{n-1} R_{j}\left(e^{-\lambda_k t}\right)e^{-\lambda_j \tau} \right) d\tau .
		\end{equation*}  
		\begin{itemize}
			\item[1)]   For $k < n$: The expression inside the parentheses evaluates to the everywhere zero function (using induction), and hence the functional evaluates to zero by \eqref{Sampling at Infinity}.
			
			\item[2)]  For $k = n$: The functional evaluates to unity since the summation inside the integral evaluates to the everywhere zero function (using induction) reducing the expression to:
				\begin{equation}
					\int_{[0\comma \infty]}e^{-t}\delta\left(e^{-t}\right)e^{\left(\lambda_n - \lambda_k\right)t} \, dt = \int_{[0\comma \infty]}e^{-t}\delta\left(e^{-t}\right) \, dt = 1
				\end{equation}
				which follows from \eqref{Sampling at Infinity} and $\lambda_n = \lambda_k$.
			\item[3)] For $k > n$: The summation inside the integral evaluates to the everywhere zero function reducing the functional to:
				\begin{equation}
					\int_{[0\comma \infty]}e^{-t}\delta\left(e^{-t}\right)e^{\left(\lambda_n - \lambda_k\right)t} \, dt = 0
				\end{equation}
				which follows from \eqref{Sampling at Infinity} and $\lambda_n < \lambda_k$.
		\end{itemize}
		Making this claim precise requires only standard analysis in order to invoke an appropriate representation theorem. It is immediate from the previous discussion that the algorithm proposed in this paper is effectively evaluating the biorthogonal linear functionals associated with $\{e^{-\lambda_{n}t} \colon n \in \N \}$, where $\Lambda$ is additionally determined at runtime.

	\subsection{Connections to Monomial Approximation and Taylor Series} 
		We now elucidate several relationships between a modification of Taylor approximation using monomials (sans the constant monomial) and our transient signal decomposition algorithm when $\Lambda \subseteq \N$. In particular, consider a given function, $f \colon [0\comma 1] \rightarrow \R$, in the space defined by the span of the monomials $\left\{ z^n \colon n \in \N \right\}$, written as:
		\begin{equation}
			f(z) = \sum_{n \in \N}\alpha_{n} z^n .
		\end{equation}
		Analogous to Subsection~\ref{Functional Interpretation}, we assert that $f$ is decomposable as: 
		\begin{equation}
			f(z) = \sum_{n\in\mathbb{N}} Q_{n}(f) z^{n}
		\end{equation}
		where the linear functionals, $Q_n  :  \text{span} \left\{z^n  \colon  n \in \N \right\}  \rightarrow  \R$, $n \in \N$, are:
		\begin{equation} \label{Eq:Q Definition}
			Q_{n}(f) \triangleq \frac{(-1)^{n}}{\Gamma(n)}\int_{[0\comma 1]}\delta^{(n)}(z)f(z) \, dz
		\end{equation}
		where $\delta^{(n)}(\cdot)$ denotes the $n$th distributional derivative of the Dirac delta function. A proof of this claim argues that the linear functionals in \eqref{Eq:Q Definition} also satisfy a set of conditions resembling biorthogonality: 
		\begin{equation} \label{Eq:Monomial Biorthogonal}
			\forall n\comma k \in \N\comma  \enspace Q_{n}(z^{k}) = \Delta[n-k] = \left\{ \begin{array}{cc}1\comma  & n=k\\0\comma  & n\neq k\end{array}\right. .
		\end{equation}
		The possible scenarios of \eqref{Eq:Monomial Biorthogonal} are summarized next using:
		\begin{equation*}
			Q_{n}(z^{k}) = \frac{(-1)^{n}}{\Gamma(n)} \int_{[0\comma 1]} \delta^{(n)}(z)z^{k} \, dz .
		\end{equation*}  
		\begin{itemize}
			\item[1)]   For $k < n$: The functional evaluates to zero since the monomial $z^k$ is annihilated by the distributional derivative resulting in a sampling of the everywhere zero function by the sifting property in \eqref{Sampling at Infinity} on the left.
			
			\item[2)]  For $k = n$: The functional evaluates to unity since taking $n$ derivatives of the monomial $z^n$ yields an appropriately scaled constant valued function.
			
			\item[3)]  For $k > n$: The functional evaluates to zero despite the fact that the distributional derivative does not annihilate the monomial $z^k$, because the sampling occurs at $z=0$ for which all non-constant monomials are zero valued. 
		\end{itemize}

		These scenarios coincide with the development in the previous subsection, and by extension, the behavior of the algorithm proposed in this paper. We conclude with the observation that these relationships are characterized by mapping \eqref{Eq:Q Definition} into \eqref{Eq:R Definition} by first recasting the linear functionals, $\left\{Q_{n} : n \in \N\right\}$, as:
		\begin{equation}
			Q_{n}(f) = \int_{[0\comma 1]} \delta(z)z^{-n}\left( f(z) - \sum_{k = 1}^{n-1} Q_{k}(f)z^k \right) \, dz
		\end{equation}
		and then translating the domain $[0\comma 1]$ to $[ 0 \comma  \infty ]$ by substituting  $z = e^{-t}$.  Indeed, the appropriate sampling of higher order monomials $(k>n)$ parallels the limiting transient behavior of the exponential decay rate  $\lambda_n - \lambda_k$ while the annihilation or subtraction of lower order monomials is reminiscent of stage 3 in the proposed algorithm.

\cleardoublepage
\balance
\bibliographystyle{IEEEtran}
\bibliography{refs}

\end{document}